\documentclass[12pt]{amsart}
\usepackage{amssymb}
\usepackage{amsmath, amscd}
\usepackage{amsthm}
\usepackage{xcolor}
\usepackage{ulem}

\newtheorem{theorem}{Theorem}[section]

\newtheorem{proposition}[theorem]{Proposition}

\newtheorem{corollary}[theorem]{Corollary}
\theoremstyle{definition}

\newtheorem{example}[theorem]{Example}

\newtheorem{remark}[theorem]{Remark}

\def\Z{{\mathbb Z}}
\def\Q{{\mathbb Q}}
\def\N{{\mathbb N}}

\newcommand{\im}{{\rm Im}}

\usepackage[utf8]{inputenc}
\usepackage{amssymb}
\usepackage{amsmath, amscd}
\usepackage{amsthm}
\usepackage{xcolor}

\begin{document}

\author[A.R. Chekhlov]{Andrey R. Chekhlov}
\address{Department of Mathematics and Mechanics, Tomsk State University, 634050 Tomsk, Russia.}
\email{cheklov@math.tsu.ru; a.r.che@yandex.ru}
\author[P.V. Danchev]{Peter V. Danchev}
\address{Institute of Mathematics and Informatics, Bulgarian Academy of Sciences \\ "Acad. G. Bonchev" str., bl. 8, 1113 Sofia, Bulgaria.}
\email{danchev@math.bas.bg; pvdanchev@yahoo.com}
\author[P.W. Keef]{Patrick W. Keef}
\address{Department of Mathematics, Whitman College, 345 Boyer Avenue, Walla Walla, WA, 99362, United States of America.}
\email{keef@whitman.edu}

\title[Strongly co-Hopfian Abelian Groups]{Strongly and Uniformly Strongly \\ co-Hopfian Abelian Groups}
\keywords{co-Hopfian group, (uniformly) strongly co-Hopfian group, algebraically compact group, cotorsion group}
\subjclass[2010]{20K10, 20K20, 20K21, 20K30}

\maketitle

\begin{abstract} We consider the so-called {\it strongly co-Hopfian} and {\it uniformly strongly co-Hopfian} Abelian groups, significantly generalizing some important results due to Abdelalim in the J. Math. Analysis (2015). Specifically, we prove that any strongly co-Hopfian group is a direct sum of an sp-group and a divisible group, both of which are strongly co-Hopfian. We also show that a group whose maximal torsion subgroup and corresponding torsion-free factor are both strongly co-Hopfian will also be strongly co-Hopfian. We provide several examples demonstrating that the converse of this statement does {\it not} generally hold, thus illustrating that the structure of genuinely mixed strongly co-Hopfian groups is rather complicated and does {\it not} entirely depend on the structure of its maximal torsion subgroup. We also establish that a strongly co-Hopfian group is cotorsion exactly when it is algebraically compact and, particularly, a reduced (adjusted) cotorsion group is strongly co-Hopfian only when its maximal torsion subgroup is strongly co-Hopfian. Additionally, we demonstrate that a strongly co-Hopfian group is uniformly strongly co-Hopfian exactly when its maximal torsion subgroup is strongly co-Hopfian.
\end{abstract}

\section{Introduction and Motivation}

All groups considered and examined in the present research article are {\it additive} and {\it Abelian} as the basic notation and terminology are in agreement with the classical books \cite{F} and \cite{Fu}. The symbol $G$ will always denote some group, and $T$ will denote its {\it maximal torsion subgroup} or shortly just {\it torsion subgroup}, which is often termed as the {\it torsion part}, as well. If $H$ is some other group, we will typically denote its torsion subgroup by $T_H$ for clarity. Similarly, for a prime $p$, we will denote the $p$-torsion subgroup of $G$ by $T_p$, so that the equality $T=\bigoplus_p T_p$ is always fulfilled.

\medskip

It is well known in the existing literature that a group $G$ is said to be {\it co-Hopfian} if, and only if, it is {\it not} isomorphic to any of its proper subgroups. In a related vein, imitating \cite{A}, a group $G$ is called {\it strongly co-Hopfian} (abbreviated in the sequel as {\it Sco-H} for short) if the chain $$\im f\geq\im f^2\geq\dots \geq \im f^n \geq \dots$$ is stationary (i.e., it stabilizes at some place) for every endomorphism $f$ of $G$.

\medskip

Note that all Artinian modules satisfy this definition. So, the same holds for all groups having the minimality condition with respect to subgroups (i.e., such are the finite direct sums of co-cyclic groups).

Since the equality $\im f^n=\im f^{n+1}$ for some integer $n>0$, where $f$ is an endomorphism of $G$, is obviously equivalent to the condition $\im f^n+\ker f^n=G$ (cf. \cite[Proposition 2.1]{A}), each Sco-H group is necessarily co-Hopfian. 

Nevertheless, another way to see this important fact is to use an indirect argument like this: If $G$ fails to be co-Hopfian, then there must be an injective homomorphism $\phi:G\to G$ that is not surjective. If $x\in G\setminus \phi(G)$, then, for all $n\in \N$, it follows that $$\phi^n(x)\in \phi^n(G)\setminus \phi^{n+1}(G)$$ yielding $\phi^n(G)\ne \phi^{n+1}(G)$, and hence $G$ is not Sco-H, as suspected.

It is worth mention that in \cite{A} Sco-H groups were fully characterized in the classes of torsion groups, torsion-free groups and algebraic compact groups, plus other interesting properties of Sco-H groups were additionally considered. Concretely, it was shown there that strong co-Hopficity is closed under taking direct summands (see \cite[Proposition 2.2]{A}) as well as that reduced Sco-H $p$-groups are always finite (see \cite[Proposition 2.3]{A}). Furthermore, in \cite[Theorem 2.1]{A} reduced torsion Sco-H groups were classified from which it follows at once that subgroups of reduced Sco-H groups inherit this property, i.e., they are again reduced Sco-H, and in \cite[Theorem 2.2]{A} this classification was extended to an arbitrary torsion group. Next, in \cite[Proposition 3.1]{A}, the author showed that any divisible group is Sco-H precisely when all ranks are bounded by some fixed natural number, and in \cite[Proposition 3.2]{A} he proved that every torsion-free group is Sco-H uniquely when it is divisible of finite rank. A culmination of his research exploration is the description in \cite[Theorem 3.1]{A} of each reduced algebraically compact Sco-H groups in terms of their $p$-basic subgroups for all primes $p$, that fact will be substantially generalized in Theorem~\ref{cotorsion} to the so-called {\it cotorsion groups}.

\medskip

Mimicking \cite{KTT}, recall also that a mixed group $G$ with an infinite number of non-zero $p$-components $T_p$ is called an {\it sp-group}, provided that $G$ is a pure subgroup in the cartesian direct product $\prod_p T_p$. In that case, it is worthwhile remembering that as the torsion subgroup $T$ of $G$ is the corresponding direct sum $\bigoplus_p T_p$, the factor-group $G/T$ is always divisible. For our convenience, we will also employ this terminology when $T_p$ is only non-zero for a finite number of primes; in this case, $G=T=\bigoplus_p T_p$ is a finite direct sum and again $G/T$ is divisible, since it is, in fact, $\{0\}$.

\medskip

Explicitly, the main points which motivate our work are the following four bullets:

\medskip

\noindent $\bullet$ If $T$ and $G/T$ are both Sco-H groups, then $G$ is an Sco-H group (Theorem~\ref{char}).

\medskip

\noindent $\bullet$ The converse of the last result does {\it not} hold in general, because if $G$ is an Sco-H group, then we cannot conclude that either $T$ or $G/T$ is an Sco-H group (Examples~\ref{ex0}-\ref{ex3}).

\medskip

\noindent $\bullet$ If $G$ is a reduced adjusted cotorsion group, then $G$ is Sco-H if, and only if, $T$ is Sco-H. (Corollary~\ref{cotor}).

\medskip

\noindent $\bullet$ If $G$ is an Sco-H group, then $G$ is uniformly Sco-H if, and only if, $T$ is Sco-H (Theorem~\ref{strongly}).

\section{Principal Results}

Our first pivotal technicality is the following one.

\begin{proposition}\label{splitinvariant}
Suppose $G=A\oplus B$ is a group for which $\mathrm{Hom}(B, A)=\{0\}$. Then, $G$ is Sco-H if and only if $A$ and $B$ are both Sco-H.
\end{proposition}

\begin{proof}
Necessity follows at once from the general fact that an arbitrary direct summand of an Sco-H group retains that property (see \cite[Proposition 2.2]{A}).

Regarding sufficiency, suppose $\phi:G\to G$ is an endomorphism. Clearly, $B$ is fully invariant in $G$.
Therefore, since $B$ is Sco-H, for some $n\in \N$ we have $\phi^n(B)=\phi^{n+1}(B)$. In addition, $\phi$ induces an endomorphism $\overline \phi:G/B\to G/B$, and since $G/B\cong A$ is also assumed to be Sco-H, it follows that for some $m\in \N$ that $\overline \phi^{\,m}(G/B)=\overline \phi^{{\,m}+1}(G/B)$; i.e., $\phi^m(A)+B=\phi^{m+1}(A)+B$.  Consequently,
\begin{align*}
\phi^{n+m}(G) &= \phi^n \phi^m(A+B)\cr
              &\subseteq \phi^n(\phi^m(A) +B)\cr
              &=     \phi^n (\phi^{m+1}(A) +B)\cr
              &=\phi^{n+m+1}(A) + \phi^n (B) \cr
              &= \phi^{n+m+1}(A)+ \phi^{n+m+1}(B)\cr
              &= \phi^{n+m+1}(G)\cr
              &\subseteq \phi^{n+m}(G)\cr
\end{align*}
Thus, $\phi^{n+m}(G)=\phi^{n+m+1}(G)$, completing the argument.
\end{proof}

A useful consequence is the following.

\begin{corollary}\label{reduction}
Suppose $G=R\oplus D$, where $R$ is reduced and $D$ is divisible. Then, $G$ is Sco-H if and only if $R$ is Sco-H and there is an $n_0\in \N$ such that $r_0(D)\leq n_0$ and $r_p(D)\leq n_0$ for all primes $p$.
\end{corollary}

\begin{proof}
Since $\mathrm{Hom}(D, R)=\{0\}$ holds, Proposition~\ref{splitinvariant} tells us that $G$ is Sco-H if, and only if, both $R$ and $D$ have that property. And utilizing \cite[Proposition~3.1]{A}, $D$ is Sco-H if, and only if, such an $n_0$ exists.
\end{proof}

So, by Corollary~\ref{reduction}, a complete description of the Sco-H groups quickly reduces to the case of reduced groups. The following helpful observation further restricts our attention to sp-groups with finite $p$-torsion components.

\begin{proposition}\label{embedding}
Suppose $G$ is a reduced Sco-H group with torsion $T=\bigoplus_p T_p$. Then, each $T_p$ is finite and $G\cong \hat G$, where $T\leq \hat G \leq \prod_p T_p$ and $\hat G/T$ is (torsion-free) divisible. In other words, $G$ is an sp-group, provided the number of its non-zero $p$-components is infinite.
\end{proposition}

\begin{proof}
Fix some prime $p$. Considering multiplication by $p$ as an endomorphism of $G$, it follows that there is an $n\in \N$ such that $p^n G=p^{n+1}G$. Since $T_p$ is pure in $G$, this implies that $p^n T_p=p^{n+1}T_p$. Therefore, $p^n T_p$ must be $p$-divisible, and since it is a $p$-group, it is also $q$-divisible for all primes $q\ne p$. Hence, $p^n T_p$ is divisible, whence $\{0\}$, since $G$ is reduced.

Since $p^n T_p=\{0\}$, and $T_p$ is pure in $G$, it follows that $G=T_p\oplus H_p$ for some $H_p\leq G$. If $T_p$ is infinite, then since it is bounded, it would have a direct summand of the form $\Z(p^j)^{(\omega)}=:X$. However, such a group $X$ is, evidently, not co-Hopfian, so that $G$ would also fail to be co-Hopfian, which contradicts the assumption that it is Sco-H. So, each subgroup $T_p$ is finite indeed.

Apparently, $H_p$ has no $p$-torsion, so that multiplication by $p$ is an injection on $H_p$. If it failed to be a surjection, we could conclude that $H_p$ is not co-Hopfian, so that $G$ is also not co-Hopfian, which contradicts that $G$ is Sco-H. Thus, multiplication by $p$ is an automorphism of $H_p$, which forces that $H_p = p^{\omega}G$. Therefore, $T_p$ is dense in $G$ in the $p$-adic topology. So, we can conclude that $T$ is dense in $G$ in the $p$-adic topology for all primes $p$, i.e., it is dense in $G$ in the $\Z$-adic topology. In particular, $G/T$ must be divisible.

Note, furthermore, that $\prod_p T_p$ will be the completion of $T$ in the $\Z$-adic topology, so there is a homomorphism $G\to \prod_p T_p$ extending $T\hookrightarrow \prod_p T_p$, whose kernel is the first Ulm subgroup, $G^1=\cap_{n\geq 1} ~ nG = \cap_{\forall p} ~ p^{\omega}G = \cap_{\forall p}~(\cap_{m\geq 1} ~ p^mG)$, say. We now claim that $G^1$ must be divisible, and since $G$ is reduced, this means that $G^1=\{0\}$, enabling the desired embedding. To this end, suppose $x\in G^1$ and $p$ is an arbitrary prime. As $x\in G^1\subseteq p^\omega G=H_p$, it follows that $x\in H_p$. In the above decomposition, $G=T_p\oplus H_p$, and because multiplication by $p$ is an automorphism of $H_p$, there is a $y\in H_p$ such that $py=x$. Notice that, since $y\in H_p$, we have that $y\in p^\omega G$ and, for all $q\ne p$, since $(p,q)=1$, $x$ and $y$ will have the same (infinite) $q$-height. Finally, $y\in G^1$, so that $G^1$ is really divisible, and hence $\{0\}$, as claimed, finishing the argumentation.
\end{proof}

However, in contrast to Example~\ref{ex1} alluded to below, the next statement, which is a direct consequence of the previous Proposition~\ref{embedding}, shows the curious fact that the $p$-component of an arbitrary Sco-H group is again Sco-H for any prime $p$: In fact, each $p$-basic subgroup of an Sco-H group is a $p$-group as well as each $p$-component of an Sco-H group is its direct summand, and thus too an Sco-H group.

So, we thereby come to the following satisfactory observation concerning the mixed case: In virtue of the preceding Corollary~\ref{reduction}, and because the divisible Sco-H-groups were described in \cite[Proposition 3.1]{A}, the whole Sco-H group structure depends on its reduced part, especially when it is Sco-H, so that the study of Sco-H groups generally limits to reduced Sco-H groups.

Besides, one hopefully observes that, if an Sco-H group $G$ has only a finite number of non-zero $p$-components, then we may write $$G=G_{p_1}\oplus\dots\oplus G_{p_n}\oplus H,$$ where, if $H\neq\{0\}$, then $H$ is a non-trivial divisible torsion-free group of finite rank, and so the structure of such Sco-H groups $G$ is totally determined by usage of \cite{A}. Therefore, in most of our considerations, we will hereafter consider {\bf only} mixed Sco-H groups having an infinite number of non-zero $p$-components.

\medskip

Next, recall the following two key results from \cite{A}.

\medskip

(A) (Theorem 2.1) If $T$ is a reduced torsion group, then $T$ is Sco-H if, and only if, there is an $e\in \N$ such that $\mathrm{ord}(T_p)\leq p^e$ for all primes $p$. In particular, an arbitrary subgroup of a reduced torsion Sco-H group remains Sco-H.

\medskip

(B) (Proposition 3.2) If $A$ is a torsion-free group, then $A$ is Sco-H if, and only if, it is divisible of finite rank.

\medskip

These facts bring us to our most important result saying the following.

\begin{theorem}\label{char} Suppose $G$ is an arbitrary group. If both $T$ and $G/T$ are Sco-H groups, then so is $G$ as well.
\end{theorem}

\begin{proof}
Suppose $G=R\oplus D$, where $R$ is reduced and $D$ is divisible. It is straightforward to verify that the result holds if, and only if, it holds for $R$ and $D$ individually. In addition, using \cite[Proposition~3.1]{A}, it is easy to check that the result holds for $D$. Therefore, there is no loss of generality in assuming that $G=R$ is reduced.

By point (B) above, we know that $\overline G:=G/T$ will be a torsion-free finite rank divisible group, and thus
we will induct on its rank, i.e., $n:=\mathrm{dim}(\overline G)$, its dimension as a vector space over $\Q$.

If, for a moment, $n=0$, then it follows that $G/T=\overline{\{0\}}$, so that $G=T$, which is assumed to be Sco-H. So, assume the result is true whenever $H$ satisfies our hypotheses and $\overline H:=H/T_H$ has dimension strictly less than $n=\dim (\overline G)>0$.

Suppose $\phi:G\to G$ is an endomorphism. So, $\phi$ induces an endomorphism $\overline \phi:\overline G\to \overline G$. Set $H:=\phi(G)$. Note that $T_H=T\cap H\leq T$. By point (A) above, $T_H$ will also be Sco-H. In addition,
$$\overline H=H/(T\cap H)\cong (H+T)/T =\overline \phi(\overline G)\leq \overline G$$
is divisible of finite rank with $\dim (\overline H)\leq n$.

\medskip

We now differ two basic cases:

\medskip

\noindent{\bf Case 1.} $\dim (\overline H)< n$: So, by induction, $H$ is also Sco-H. It is routine to see that $\phi$ restricts to an endomorphism of $H=\phi(G)$. Thus, there is a $k\in \N$ such that $\phi^k(H)=\phi^{k+1}(H)$. And if $j=k+1$, it follows that $\phi^j(G)=\phi^{j+1}(G)$, as requested.

\medskip

\noindent{\bf Case 2.} $\dim (\overline H)=n$: It follows that $\overline \phi(\overline G)=\overline G$. And since $\overline G$ is finite dimensional (with dimension exactly $n$), we can conclude that $\overline \phi:\overline G\to \overline G$ is an automorphism. Therefore, for all $j\in \N$, $\overline \phi^j$ will also be an automorphism of $\overline G$. This means that, if $j\in \N$, $x\in G$ and $\phi^j(x)\in T$, then $x\in T$. Thereby, we have shown that, for all $j\in \N$, the equality $T\cap \phi^j (G)=\phi^j(T)$ is valid.

However, since $T$ is Sco-H, there is a $k\in \N$ such that
$$       T\cap \phi^k (G)=\phi^k(T)=\phi^{k+1}(T)= T\cap \phi^{k+1} (G).
$$

And since
$$(\phi^k (G)+T)/T=\phi^{k} (\overline G)=\overline G= \phi^{k+1} (\overline G)=(\phi^{k+1} (G)+T)/T,$$ there is a commutative diagram with short exact rows:

\medskip

$$\begin{CD}
0 @>>> T\cap \phi^{k+1} (G) @>>> \phi^{k+1} (G) @>>>  \overline G @>>> 0 \\
    @.     @|                   @VV\subseteq V               @| \\
0 @>>> T\cap \phi^k (G)@>>> \phi^{k} (G) @>>> \overline G @>>> 0 \\
\end{CD}$$

\medskip

Finally, $\phi^{k} (G)=\phi^{k+1} (G)$, as requested, completing the proof.
\end{proof}

The major theme of the rest of our research is an investigation of when the converse of the last result will hold. In this aspect, the following immediate consequence of either Theorem~\ref{char} or Proposition~\ref{splitinvariant} gives an affirmative answer to this in the case of \textit{splitting} mixed groups.

\begin{corollary}
Suppose $G=A\oplus T$ is a splitting mixed group. Then, $G$ is Sco-H if and only if both $T$ and $G/T \cong A$ are Sco-H.
\end{corollary}

The following necessary and sufficient condition is a slight restatement of \cite[Theorem~3.1]{A}.

\begin{theorem}\label{3.1}
The reduced algebraically compact group $G$ is Sco-H if and only if $T=\oplus_p T_p$ is Sco-H and $G\cong\prod_p T_p$.
\end{theorem}

So, with the aid of Proposition~\ref{reduction}, we can slightly generalize the last result as follows:

\begin{corollary}
Suppose $G=R\oplus D$ is an algebraically compact group, where $R$ is reduced with torsion $T_R=\bigoplus_p T_p$ and $D$ is divisible. Then, $G$ is Sco-H if and only if $T_R$ is Sco-H, $R=\prod_p T_p$ and there is an $n_0\in \N$ such that $r_0(D)\leq n_0$ and $r_p(D)\leq n_0$ for all primes $p$.
\end{corollary}

The following generalizes the above result to the class of arbitrary cotorsion groups.

\begin{theorem}\label{cotorsion}
Suppose $G$ is an Sco-H group. Then, $G$ is cotorsion if and only if it is algebraically compact.
\end{theorem}

\begin{proof}
It is easy to inspect that a group is either cotorsion or algebraically compact if, and only if, its reduced part has that property. There is, therefore, no loss of generality in assuming that $G$ is reduced. Certainly, if $G$ is algebraically compact, then it is cotorsion (since any algebraically compact group is cotorsion as formulated in \cite{F,Fu}).

Conversely, suppose $G$ is cotorsion (and reduced Sco-H). If $T=\oplus_p T_p$ is the torsion subgroup of $G$, then in view of Proposition~\ref{embedding}, each $T_p$ is finite, and $G\cong \hat G$, where $T\leq \hat G\leq \prod_p T_p:=P$ and $\hat G/T$ is divisible and torsion-free. Hence, $\hat G/T$ will be a divisible direct summand of
$P/T$, which is also torsion-free divisible. Therefore, $P/\hat G\cong (P/T)/(\hat G/T)$ will also be torsion-free and divisible whence, by the definition of cotorsion groups, $\hat G$ must be a direct summand of $P$; i.e., $P=\hat G\oplus E$, where $E\cong P/\hat G$ is torsion-free and divisible. However, since $P$ is reduced, it automatically follows that $E=\{0\}$, i.e., $G\cong \hat G\cong P$ is algebraically compact, as pursued.
\end{proof}

As a direct consequence of Theorems~\ref{cotorsion} and \ref{3.1}, we derive:

\begin{corollary}\label{cotor}
If $G$ is a reduced adjusted cotorsion group, then $G$ is Sco-H if and only if $T$ is Sco-H.
\end{corollary}

We now say a group $G$ is {\it uniformly Sco-H} if there is a fixed $m\in \N$ such that, for all endomorphisms $\phi:G\to G$, we have $\phi^m(G)=\phi^{m+1}(G)$; we call such an integer $m$ satisfying this definition an {\it Sco-H bound} for $G$. It is quite obvious that a group that is uniformly Sco-H is always Sco-H, and we want to investigate when the reciprocal implication is also valid. To that target, it is easy to check that a direct summand of a uniformly Sco-H group remains so. The proof of Proposition~\ref{splitinvariant} also supplies the following partial converse.

\begin{proposition}\label{uniform1}
Suppose $G=A\oplus B$, where $\mathrm{Hom}(B,A)=\{0\}$. Then, $G$ is uniformly Sco-H if and only if $A$ and $B$ are both uniformly Sco-H.
\end{proposition}

In the context of the last result, if $m$ is an Sco-H bound for $A$ and $n$ is an Sco-H bound for $B$, then $m+n$ is an Sco-H bound for $A\oplus B$. Also, Proposition~\ref{uniform1} has the following immediate consequence.

\begin{corollary}\label{uniform7}
Suppose $G=A\oplus D$, where $D$ is divisible and $A$ is reduced. Then, $G$ is uniformly Sco-H if and only if both $A$ and $D$ have that property.
\end{corollary}

We note the following assertion, which gives some elementary examples of Sco-H groups that are, in fact, uniformly Sco-H.

\begin{proposition}\label{uniform8}
Suppose $G$ is an Sco-H group. If $G$ is a member of one of the following classes, then it is uniformly Sco-H.

(a) $G=D$ is divisible;

(b) $G=T$ is torsion.
\end{proposition}

\begin{proof}
For (a), suppose $G=A\oplus T$, where $A$ is torsion-free. Consulting with Proposition~\ref{uniform1}, $G$ is uniformly Sco-H if, and only if, both $A$ are $T$ have this property. If $G=A$ is torsion free, then it must have finite rank, say $m$. It easily follows that $m$ will also be an Sco-H bound for $G$, so that it is uniformly Sco-H. If $G=T$ is torsion, and each $p$-component $T_p$ has $p$-rank at most $n$, then it also follows automatically that $n$ will be an Sco-H bound for $G$.

Regarding (b), by (a) we may assume that $G$ is reduced, so there is a fixed natural number $e$ such that $\mathrm{card}(T_p)\leq p^e$ for every prime $p$. It is then plainly seen that $e$ will be an Sco-H bound for $G$.
\end{proof}

The following criterion gives a complete description of when a group which is Sco-H is, in fact, uniformly Sco-H.

\begin{theorem}\label{strongly}
Suppose $G$ is a group that is Sco-H. Then, $G$ is uniformly Sco-H if and only if $T$ is Sco-H.
\end{theorem}

\begin{proof} Let $G=A\oplus D$, where $A$ is reduced and $D$ is divisible. It follows from Corollary~\ref{uniform7} and Proposition~\ref{uniform8}(a) that $G$ is uniformly Sco-H if, and only if, $A$ has that property, and $T=T_A\oplus T_D$ is uniformly Sco-H if, and only if, $T_A$ has that property. Therefore, there is no loss of generality in assuming that $G=A$ is reduced, so that we can view $G$ as a pure subgroup of $P=\prod_p T_p$ containing $T$, where each $T_p$ is finite.

Regarding necessity, suppose that $G$ uniformly Sco-H; let $m$ be an Sco-H bound for $G$. If $p$ is any prime, we contend that $\mathrm{card}(T_p)\leq p^{(m+1)^2}$, giving the desired implication. If $r_p$ is the rank of $T_p$, then it is easy to see that there is an endomorphism $\phi:T_p\to T_p$ such that $\phi^{r_p-1}(G)\ne 0=\phi^{r_p}(G)$. This shows that $r_p-1\leq m$, i.e., $r_p\leq m+1$. Similarly, if $b_p$ is the smallest non-negative integer such that $p^{b_p} T_p=\{0\}$, then it follows that $b_p-1\leq m$, i.e., $b_p\leq m+1$. Therefore, since $T_p$ is the direct sum of $r_p$ cyclic groups of order at most $p^{b_p}$, we deduce
$$
\mathrm{card}(T_p)\leq (p^{b_p})^{r_p}= p^{b_pr_p}  \leq p^{(m+1)^2},
$$
thus establishing one implication.

\medskip

Turning to sufficiency, suppose $T$ is Sco-H, i.e., for some $e\in \N$, we have $\mathrm{card}(T_p)\leq p^e$ for all primes $p$. Let $m$ be some Sco-H bound for $T$. We will show that $m$ is also an Sco-H bound for $G$. To that goal, let $\phi:G\to G$ be some endomorphism. We first establish the following assertion.

\medskip

{\bf Claim:} Suppose $\gamma: G\to G$ is an endomorphism with image $\gamma(G)=X$. Then, $T_X=\gamma(T)$ and $X/T_X$ is divisible:

\medskip

For simplicity, let $\gamma(T)=S\leq  T_X$, and let $S_p$ be its $p$-component. There are clearly surjections $$G/T\to X/S\to X/T_X,$$ and since $G/T$ is divisible, so are $X/S$ and $X/T_X$. So, we only need to show that $T_X=\phi(T)=S$.

Note that $G/S$ has torsion $T/S\cong \bigoplus_p (T_p/S_p)$, and since each $T_p$ is finite, it follows that $T/S$ is reduced. As $X/S$ is divisible and contained in $G/S$, we can conclude that $X/S$ is torsion-free. Therefore, it must be that $T_X=S=\phi(T)$, establishing the stated Claim.

\medskip

If $k<\omega$, then let $G^k:=\phi^k(G)$, so we arrive at the descending chain of subgroups $$G\supseteq G^1\supseteq G^2\supseteq \cdots.$$ According to the last Claim with $\gamma$ equaling $\phi^k$, it follows at once that if $[G^k]_T=: T^k$, then $T^k=\phi^k(T)$ and $G^k/T^k$ is divisible.

However, since $m$ is an Sco-H bound for $T$, it follows that $T^m=T^{m+1}$; thus, we need to show $G^m=G^{m+1}$. For simplicity, denote $A=G^m$ and $B=G^{m+1}=\phi(A)$, and also designate $T^m=T^{m+1}$ by $S$; so, $B\leq A$ and we need to prove only $A=B$.

To substantiate this, since the first Ulm subgroups satisfy $A^1\leq P^1=\{0\}$, we can identify $A$ with a pure subgroup of $\prod_p S_p$ containing $S=\bigoplus_p S_p$ (i.e., it is an sp-group). So, the homomorphism $\phi:A\to A$, whose image is $B$, is entirely determined by its restriction to $S$. However, since each $S_p$ is finite, and $\phi$ induces a surjective endomorphism $S\to S$, we can infer that $\phi$ restricts to an automorphism of $S$.

Note that $\phi$ also extends to a homomorphism $\prod_p S_p\to \prod_p S_p$. Since $\phi$ is an automorphism on $S$, it is too an automorphism on $\prod_p S_p$, so that the homomorphism $\phi:A\to B$ is not only surjective, but also injective; in other words, $\phi: \phi^m(G)\to \phi^{m+1}(G)$ is an isomorphism.

Assume now by way of contradiction that $\phi^m(G)\ne \phi^{m+1}(G)$. Then, for all $j<\omega$, since $\phi^j$ will also be an injection when restricted to $A=\phi^m(G)$, it follows that, for all such $j$, we discover $\phi^{m+j}(G)\ne \phi^{m+j+1}(G)$, which obviously contradicts that $G$ is Sco-H, and completes the all details in our proof.
\end{proof}

We now wish to consider some particular examples as recently discussed. Again, Theorem~\ref{char} says that, if $G$ is some group and we know that both its torsion and the corresponding torsion-free factor are both Sco-H, then we can conclude that $G$ is so, as well; in other words, these two conditions are only \textit{sufficient} to guarantee that $G$ be Sco-H. That is why, the {\bf big} problem is still the following: Does there exist an Sco-H group $G$ such that $T$ and/or $G/T$ are {\it not} Sco-H? So, we now proceed by constructing a series of examples (see, for a more account, Examples~\ref{ex0}-\ref{ex3} stated below which give a partial solution to this general query) subsumed by the quoted above Theorem~\ref{char} and the corresponding question.

\medskip

First, we intend to show that for $G$ to be Sco-H it is \textit{not} necessary that $T$ be Sco-H, nor is it necessary that $G/T$ be Sco-H.

\begin{example}\label{ex0} There is a group $G$ such that both $G$ and $T$ are Sco-H, but $G/T$ is \textit{not} Sco-H.
\end{example}

\begin{proof} Let $T$ be Sco-H with $T_p\ne \{0\}$ for an infinite number of primes $p$; for instance, we might have $T=\bigoplus_p\Z(p)$. Then, with Theorem~\ref{3.1} at hand, its $\Z$-adic completion $G$, that is, $G=\prod_p T_p$, will be Sco-H. However, in this case, $G/T$ will be divisible of infinite rank, i.e., not co-Hopfian and so not Sco-H, as asserted.
\end{proof}

A logical question which arises is of whether this example can be amended to groups of finite torsion-free rank.

\medskip

At this point, it might be tempting to conjecture that \textit{any} Sco-H group is, in fact, uniformly Sco-H. The following construction exhibits, however, that this is {\it not} the case.

\begin{example}\label{ex1} There is a group $G$ such that both $G$ and $G/T$ are Sco-H, but $T$ is \textit{not} Sco-H.
\end{example}

\begin{proof} Let $p_1, p_2, \dots = 2, 3, 5, 7,\cdots$ be a listing of all the primes. Put
$$T:=\oplus_{n\in \N} ~ \Z(p_n^n)=\Z(2)\oplus \Z(3^2)\oplus \Z(5^3)\oplus \Z(7^4)\oplus \cdots;$$ clearly, $T$ is not an Sco-H group. Let $$P:=\prod_{n\in \N} \Z(p_n^n)$$ be the $\Z$-adic completion of $T$. Inside of $P$, let $\mathbf z=(1)_n$, i.e., the $n$th coordinate of $\mathbf z$ is $1\in \Z(p_n^n)$.

There is, obviously, a unique pure subgroup $G\leq P$ containing both $T$ and $z$ such that $G/T\cong \Q$. So, we need to show only that $G$ is an Sco-H group; to that end, let $\phi:G\to G$ be an endomorphism. Notice that $\phi$ is entirely determined by its restriction to $T$. Thus, there is a unique vector $(\alpha_n)_{n\in \N}$, where each $\alpha_n\in\Z(p_n^n)$, so that, for all $\mathbf{y}=(y_n)_{n\in \N}\in G$, we have
$$
         \phi(\mathbf y) = (\alpha_ny_n)_{n\in \N}.
$$
In fact, it easily follows that $(\alpha_n)_{n\in \N}=\phi(\mathbf z)$.

Note, likewise, that $\phi$ induces an endomorphism $$\overline \phi:\Q\cong G/T\to G/T\cong \Q.$$
It, therefore, follows that there is an $\frac{a}{b}\in \Q$, where $a\in \Z$, $b\in \N$ are relatively prime, such that $$(b\alpha_n)_{n\in \N}=b\phi(\mathbf z)=\phi(b\mathbf z)=a\mathbf z +\mathbf x,$$ where $\mathbf x\in T$. Choose $k\in \N$ such that the following three conditions hold:

\medskip

(1) $\mathbf x\in \bigoplus_{n<k} \Z(p_n^n):=B_k$;

\medskip

(2) If $p_n\vert b$, then $n<k$;

\medskip

(3) If $a\ne 0$ and $p_n\vert a$, then $n<k$.

\medskip

\noindent Set

$$G_k:=\left(\prod_{n\geq k} \Z(p_n^n)\right)\cap G; \ \ {\rm so},\ \ G=B_k\oplus G_k.$$

We now consider two possible cases:

\medskip

\noindent{\bf Case 1.} $a=0$: It, thereby, follows that, if $n\geq k$, then since $(p_n,b)=1$ and $b \alpha_n =0$, it must be that $\alpha_n=0$. Consequently, $\phi(G_k)=\{0\}$. Hence, each subgroup $\phi^j(G)=\phi^j(B_k)$ is finite. This, evidently, ensures for some $j\in \N$ that we must have the equality $\phi^j(G)=\phi^{j+1}(G)$, as required.

\medskip

\noindent{\bf Case 2.} $a\ne 0$: Note, for all $n\geq k$, that the elements $a$ and $b$ are units in $\Z(p^n_n)$. Therefore, multiplication by $\alpha_n$ agrees with multiplication by $\frac{a}{b}$, which is an automorphism of $\Z(p^n_n)$. So, the map $\phi$ restricts to an automorphism on $T_k:=\bigoplus_{n\geq k} \Z(p^n_n)$, namely multiplication by $\frac{a}{b}$ on each factor. In addition, the map $\phi$ induces multiplication by $\frac{a}{b}$ on $$G_k/T_k\cong G/T\cong \Q.$$ Since this map is also an automorphism, one readily observes that the map $\phi$ restricts to an automorphism on $G_k$.

Thus, for all $j\in \N$, we deduce that $G_k = \phi^j(G_k)\leq \phi^j(G)$, so that $$\phi^j(G)=(B_k\cap \phi^j(G))\oplus G_k.$$ Since each $B_k\cap \phi^j(G)\leq B_k$ is finite, we conclude for some $j$ that $$B_k\cap \phi^j(G)=B_k\cap \phi^{j+1}(G),$$ whence $\phi^j(G)=\phi^{j+1}(G)$, as required.
\end{proof}

\begin{remark}
With not too much effort, it can be seen that the group $G$ in the last example has a ring structure such that any group endomorphism $\phi:G\to G$ is a multiplication by some $\alpha\in G$; in other words, $G$ is an \textit{E-ring} as defined in \cite{S}.

Can the above construction be expanded by finding an sp-group of torsion-free rank greater than 1 which is Sco-H, but its torsion part is {\it not} Sco-H, i.e., $G$ is Sco-H, but {\it not} uniformly Sco-H? Such an example can be constructed using the following result, whose proof, based upon the techniques of the above example, is left to the interested reader.
\end{remark}

\begin{proposition}\label{uniformly11}
Let $\{p_1, p_2, \dots\}$ be an infinite collection of distinct primes and, for each $i\in \N$, let $e_i\in \N$. Consider the (unitary) ring $P=\prod_{i\in \N} \Z(p^{e_i})$, and let $G$ be a (unitary) subring of $P$ containing the ideal $T=\bigoplus_{i\in \N} \Z(p^{e_i})$. If the quotient ring $G/T$ is a field, then $G$ is an E-ring which is Sco-H as an abelian group, so that, if the $e_i$s are not uniformly bounded, then $G$ is not uniformly Sco-H.
\end{proposition}

One notes that in Example~\ref{ex1}, we are in the case of Proposition~\ref{uniformly11}, where all primes are used, $e_i=i$ for each $i\in \N$ and $G/T\cong \Q$.

\medskip

Now, one remembers that the listed above Theorem~\ref{3.1} claims that, if $G=\prod_p T_p$, where each $T_p$ is a finite $p$-group (so that $G$ is an sp-group), then $G$ is Sco-H, provided $T$ is Sco-H. Besides, Theorem~\ref{char} asserts that this is also valid when $G/T$ is divisible of finite rank. The following example, however, manifestly states that this does {\it not} hold in general.

\begin{example}\label{ex2} There is an sp-group $G$ such that $T$ is Sco-H, but $G$ is \textit{not} Sco-H (so that $G/T$ is also {\it not} Sco-H, i.e., it has infinite rank).
\end{example}

\begin{proof} In \cite[Example 20]{CK}, a reduced sp-group $G$ is constructed that is even {\it not} co-Hopfian (and hence {\it not} Sco-H) such that $T\cong \bigoplus_{n\in \N} \Z(p_n)$ for an infinite collection of distinct primes $p_n$. Apparently, this $T$ is Sco-H, assuring our claim.
\end{proof}

A natural question which arises is of whether this example can be improved to sp-groups of finite torsion-free rank?

\medskip

If $G$ is an sp-group that possesses torsion-free rank 1 such that each $T_p$ is finite, then it might be conjectured that $G$ is always Sco-H no matter whether $T$ is Sco-H or not. However, the following exhibition shows that this fails.

\begin{example}\label{ex3} There is an sp-group $G$ such that each $T_p$ is finite, $G/T$ has rank 1 (and so it is Sco-H), but neither $T$ nor $G$ is Sco-H.
\end{example}

\begin{proof} Let $\mathcal P \cup \mathcal Q$ be a partition of the collection of all primes into two infinite, disjoint subsets. Put
$$
            T_{\mathcal P}:= \bigoplus_{p\in\mathcal P} \Z(p)\leq H \leq \prod_{p\in\mathcal P} \Z(p)
$$
such that $H/T_{\mathcal P}\cong \Q$. Moreover, if $\mathcal Q=\{q_1, q_2,\dots\}$ is a listing of $\mathcal Q$, set
$$
            T_{\mathcal Q}:= \bigoplus_{n\in \N} \Z(q_n^n);
$$
so, $T_{\mathcal Q}$ is a torsion group that is surely not Sco-H in conjunction with \cite[Theorems~2.1, 2.2]{A}. 

If now, we define $G:= H\oplus T_{\mathcal Q}$, then $G$ is, immediately, an sp-group with $G/T$ of rank 1 and each $T_p$ finite, but since $T_{\mathcal Q}$ is a direct summand of both $G$ and $T$, neither one of them is Sco-H, as asked.
\end{proof}

These examples prove that in answering the question of when a reduced sp-group is actually an Sco-H group, we need to consider more than simply its torsion subgroup and corresponding torsion-free quotient individually. It is equally important, and quite complicated, to further consider the ways in which these two parts can be combined into a properly mixed group.

\medskip

We end our work with the following challenging question left unresolved by our discussions so far (compare with Proposition~\ref{splitinvariant} stated above).

\medskip

\noindent{\bf Problem.} If $A$ and $B$ are Sco-H groups, does it follows that $A\oplus B$ is also Sco-H? In particular, is the square $G\oplus G$ an Sco-H group whenever $G$ is an Sco-H group?

\medskip

It is worthy of noting that this question makes {\it no} sense for infinite direct sums as, for instance, Example~\ref{ex3} unambiguously illustrates: Indeed, every $p$-component $T_p$ is finite and thus Sco-H, but $T=\oplus_p T_p$ is {\it not} so.

\medskip

Note that in the above problem there is no loss of generality in assuming that $A$ and $B$ are both reduced sp-groups.

\bigskip

\noindent{\bf Declarations:} The authors declare {\bf not} any conflict of interests as well as {\bf no} data was used while preparing and writing the current manuscript.

\bigskip

\noindent{\bf Funding:} The work of A.R. Chekhlov is supported by the Ministry of Science and Higher
Education of Russia under agreement No. 075-02-2025-1728/2.

\end{document}